\documentclass[11pt,reqno]{article}

\usepackage[utf8]{inputenc}

\usepackage[a4paper, margin=2.5cm,top=3cm, bottom=3cm]{geometry}
\usepackage{amsmath, amsthm, amssymb, amsfonts}
\usepackage{mathtools}
\usepackage{hyperref}
\usepackage{enumitem}
\usepackage{setspace}
\usepackage{xcolor}
\usepackage{thmtools}
\usepackage{thm-restate}
\usepackage{verbatim}

\theoremstyle{plain}
\newtheorem{theorem}{Theorem}[section]
\newtheorem{lemma}[theorem]{Lemma}
\newtheorem{claim}[theorem]{Claim}
\newtheorem{proposition}[theorem]{Proposition}

\newtheorem{corollary}[theorem]{Corollary}
\newtheorem{conjecture}[theorem]{Conjecture}

\newtheorem{question}[theorem]{Question}

\theoremstyle{definition}
\newtheorem{definition}[theorem]{Definition}

\newcommand{\R}{\mathbb{R}}

\newcommand{\E}{\mathbb{E}}
\newcommand{\indicator}[1]{\mathbf{1}_{#1}}

\newcommand{\eps}{\ensuremath{\varepsilon}}

\newcommand{\uar}{u.a.r.\ }

\newcommand{\ie}{i.e.\ }

\newcommand*\samethanks[1][\value{footnote}]{\footnotemark[#1]}

\title{Long Cycles, Heavy Cycles and Cycle Decompositions in Digraphs}

\author{
	Charlotte Knierim\thanks{Department of Computer Science, ETH Zurich, Switzerland\newline \{cknierim$\vert$larcherm$\vert$anders.martinsson$\vert$andreas.noever\}@inf.ethz.ch\newline
	}
	\and
	Maxime Larcher\samethanks[1]
	\and Anders Martinsson\samethanks[1]
	\and Andreas Noever\samethanks[1]}

\hypersetup{
	colorlinks,
	linkcolor={red!60!black},
	citecolor={green!50!black},
	urlcolor={blue!80!black}
}

\newcommand{\din}{\ensuremath{\mathrm{d}^-}}
\newcommand{\dout}{\ensuremath{\mathrm{d}^+}}
\newcommand{\drem}{\ensuremath{\mathrm{rd}^+}}
\newcommand{\wrem}{\ensuremath{\mathrm{rw}^+}}
\newcommand{\win}{\ensuremath{\mathrm{w}^-}}
\newcommand{\wout}{\ensuremath{\mathrm{w}^+}}
\renewcommand{\deg}{\ensuremath{\mathrm{d}}}
\newcommand{\weight}{\ensuremath{\mathrm{w}}}
\DeclareMathOperator{\argmin}{arg\,min}

\newcommand{\Nin}{\ensuremath{\mathrm{N^-}}}
\newcommand{\outneighb}{\ensuremath{\mathrm{N^+}}}

\newcommand{\minoutdeg}{\ensuremath{{\delta^+}}}
\newcommand{\eqdef}{\ensuremath{:=}}

\newcommand\numberthis{\addtocounter{equation}{1}\tag{\theequation}}

\begin{document}
\maketitle

\begin{abstract}

Haj\'os conjectured in 1968 that every Eulerian \(n\)-vertex graph can be decomposed into at most $\lfloor (n-1)/2\rfloor$ edge-disjoint cycles. This has been confirmed for some special graph classes, but the general case remains open. In a sequence of papers by Bienia and Meyniel (1986), Dean (1986), and Bollob\'as and Scott (1996) it was analogously conjectured that every \emph{directed} Eulerian graph can be decomposed into $O(n)$ cycles.

In this paper, we show that every directed Eulerian graph can be decomposed into $O(n \log \Delta)$ disjoint cycles, thus making progress towards the conjecture by Bollob\'as and Scott. Our approach is based on finding heavy cycles in certain edge-weightings of directed graphs. As a further consequence of our techniques, we prove that for every edge-weighted digraph in which every vertex has out-weight at least $1$, there exists a cycle with weight at least $\Omega(\log \log n/{\log n})$, thus resolving a question by Bollob\'as and Scott.

\end{abstract}

\section{Introduction}
\paragraph{Cycle decompositions}
An Euler tour in a graph $G$ is a closed walk that uses every edge of $G$ exactly once. A graph which has an Euler tour is called \emph{Eulerian}. A classical result of  Euler states that an undirected graph is Eulerian if and only if it is connected and all the degrees are even. 
The notions of Euler tours and Eulerian graph naturally generalise to directed graphs. Intuitively, a graph is Eulerian if for every vertex it holds that whenever we enter this vertex there is a possibility to continue our tour. A digraph is Eulerian if and only if it is connected and for every vertex we have equally many incoming and outgoing edges. It is easy to see that by greedily removing cycles, these graphs can be decomposed into an edge-disjoint union of cycles. 

The characteristics of such decompositions have been subject to numerous conjectures in graph theory. One famous example is a conjecture by Kelly, stating that for any sufficiently large $n$, every regular $n$-vertex tournament can be decomposed into Hamilton cycles. A stronger version of this was proved by K\"uhn and Osthus in 2013. 
\begin{theorem}[K\"uhn, Osthus \cite{KO13}]
	For every $\eps>0$ there exists $n_0$ such that every $r$-regular oriented graph $G$ on $n\ge n_0$ vertices with $r\ge 3n/8+\eps n$ has a Hamilton decomposition. In particular, there exists $n_0$ such that every regular tournament on $n\ge n_0$ vertices has a Hamilton decomposition.
\end{theorem}
For undirected graphs, Haj\'os conjectured in 1968 that every Eulerian graph has a cycle decomposition with few cycles.
\begin{conjecture}[Haj\'os (cf.~\cite{B14,L68})]
Every Eulerian graph on $n$ vertices has an edge-decomposition into at most $\lfloor\frac{n-1}{2}\rfloor$ cycles.
\end{conjecture}
This conjecture has been confirmed for some special graph classes (e.g.\ planar graphs \cite{S92}, graphs with pathwidth at most six \cite{FGH17}, projective graphs \cite{FX02} and small graphs on at most 12 vertices~\cite{HNS17}). Very recently, Gir\~ao, Granet, K\"uhn and Osthus \cite{GGKO19} proved an approximate version of Haj\'os' conjecture in dense graphs. They also show that for graphs with sufficiently large minimum degree, the number of cycles required is asymptotically \(\Delta / 2\) rather than \(n/2\).

\begin{theorem}[Gir\~ao, Granet, K\"uhn, Osthus \cite{GGKO19}]
For any $ \varepsilon>0$, there exists $n_0$ such that for any Eulerian graph $G$ on $n\geq n_0$ vertices the following hold.
\begin{enumerate}[label=(\roman*)]
\item If $\delta(G)\geq \varepsilon n$, then $G$ can be decomposed into $n/2 + o( n)$ cycles.
\item If $\delta(G)\geq (1/2+\varepsilon)n$, then $G$ can be decomposed into $\Delta(G)/2+o( n)$ cycles.
\end{enumerate}
\end{theorem}

In the case of general Eulerian graphs, another natural way to approach Haj\'os' conjecture is to ask what is the minimum number of cycles we can provably decompose any $n$-vertex Eulerian graph into. Erd\"os and Gallai conjectured that every graph can be decomposed into $O(n)$ cycles and edges (see \cite{E83}, Section II, Conjecture 6), which in the case of Eulerian graphs, would translate to at most linearly many cycles. A simple upper bound comes from the connection with long cycles. A famous theorem of Erd\H{o}s and Gallai~\cite{EG59} guarantees that any graph $G$ with $n$ vertices and $m$ edges has a cycle of length $\Omega\left(\frac{m}{n}\right)$. By repeatedly removing such a cycle, we obtain a decomposition into $O(n\log n)$ cycles. Conlon, Fox and Sudakov \cite{CFS14} further improve this to $O(n\log \log \frac{m}{n})$ cycles, by considering structural properties of graphs without long cycles.  

 The question of how many cycles we need to decompose any Eulerian digraph into edge-disjoint cycles was considered in a sequence of papers by Bienia and Meyniel~\cite{BM86}, Dean~\cite{D86} and later Bollob\'as and Scott~\cite{BS96}.
\begin{conjecture}[Bollob\'as, Scott \cite{BS96}]
\label{conj:bollobas_scott}
There exists a constant $C>0$ such that every Eulerian digraph on $n$ vertices has an edge-decomposition into at most $Cn$ cycles.
\end{conjecture}

Dean observed in~\cite{D86} that the complete symmetric digraph on $4$ vertices $K^*_4$ cannot be decomposed into fewer than $4$ cycles, so the digraph made of $M$ copies of $K^*_4$, all meeting in one cut vertex, contains $3M+1$ vertices and required $4M$ cycles. This implies that if Conjecture \ref{conj:bollobas_scott} is true then $C$ has to be at least $4/3$.

\paragraph{Finding long cycles}
The connection between cycle decompositions and long cycles goes both ways. On the one hand, proving Conjecture \ref{conj:bollobas_scott} would give an Erd\H{o}s-Gallai-like result for directed graphs by considering the average number of edges per cycle. On the other hand finding cycles of length $\Theta(m/n)$ would give Conjecture \ref{conj:bollobas_scott} up to a log factor, in the same manner as for undirected graphs above. The existence of long cycles in directed graph is an interesting area on its own. As there are acyclic tournaments (where $\Theta(m/n)=\Theta(n)$) it is clear that we cannot hope to get a statement in full generality. 

It is easy to see by a greedy construction that every digraph with minimum out-degree $\minoutdeg$ has a cycle of length at least $1+\minoutdeg$. Simply start in any vertex $v_0$, and greedily build a path $v_0, v_1, \dots$ until we reach a vertex $v_k$ whose entire out-neighbourhood has already been visited. Then close a cycle by connecting $v_k$ to the earliest out-neighbour to appear in the path.

 The best known lower bound for Eulerian digraphs is due to Huang, Ma, Shapira, Sudakov and Yuster~\cite{HMSSY13}. We remark that this bound is tight, up to a constant factor, when $m=\Theta(n)$ or $m=\Theta(n^2)$.

\begin{theorem}[Huang, Ma, Shapira, Sudakov and Yuster~\cite{HMSSY13}]
	Every Eulerian digraph with $n$ vertices and $m$ edges has a cycle of length at least $1+\max\{m^2/2n^3,\lfloor\sqrt{m/n}\rfloor\}$.
\end{theorem}

By repeatedly removing long cycles, this proposition implies that any Eulerian digraph can be decomposed into $O(\sqrt{mn}\log n)$ cycles.
\paragraph{Finding heavy cycles}
A related problem is to find \emph{heavy} cycles in weighted graphs. Formally, we introduce a weight function $\weight\colon E(G)\to \mathbb{R}_{+}$ that assigns every edge a non-negative weight. We are then looking for cycles that have large weight, where the weight of a graph \(G\), denoted by $\weight(G)$, is the sum of the weights of its edges.
Regarding undirected graphs, the existence of heavy cycles was first conjectured \cite{BF89} and later proved by Bondy and Fan.
\begin{theorem}[Bondy, Fan~\cite{BF91}]
	\label{thm:bondy_fan_91}
	Let $G$ be a 2-edge-connected weighted graph on $n$ vertices. Then $G$ contains a cycle of weight at least $2\weight(G)/(n-1)$.
\end{theorem}
Turning to directed graphs, it was observed by Bollob\'as and Scott~\cite{BS96} that there is no hope to come close to this statement in this generality. Taking for instance a directed path on $n$ vertices and putting weight $0$ on the edges of the path, then adding an edge with weight $1$ between every non-adjacent vertices and directing it `backwards' with respect to the path gives a graph where the maximum weight of a cycle is of order $O(\weight(G)/n^2)$. That is, even requiring strong connectedness, the heaviest cycle need not even have asymptotically higher weight than an average edge.

 Instead, Bollob\'as and Scott~\cite{BS96} looked at some sufficient conditions for the existence of heavy cycles. They proved that if every vertex $v$ satisfies \(\wout(v) \eqdef \sum_{(v,u) \in E}{\weight(v,u)} \ge 1\), then there is a cycle $C \subseteq G$ of weight $\weight(C)\ge (24 n)^{-1/3}$. They conjectured that this result could be improved to a cycle of weight at least $2/\log_2 n$, which was done by Li and Zhang~\cite{LZ12} up to a factor of two.

\begin{theorem}[Li, Zhang \cite{LZ12}]
	Let $G$ be a weighted digraph on $n\ge 2$ vertices. If $\wout(v)\ge 1$ for every vertex $v\in V(G)$ then $G$ contains a cycle $C$ of weight $\weight(C)\ge 1/\log_2 n$.
\end{theorem}

In light of his results on weighted undirected graph, Bondy (see~\cite{BS96}) conjectured that the condition \(\wout(v) \ge 1\) implies the existence of cycles of constant weight in directed graphs. Bollob\'as and Scott~\cite{BS96} showed that this was not the case by constructing a graph in which \(\wout(v) = 1\) for every vertex \(v\) but all cycles have weight \(O \left( \log \log n / \log n \right)\). They note that the same construction was found by Spencer (see \cite{BS96}). Their construction is the following: take a rooted tree of depth \(l\) and where every non-leaf vertex has \(k = l^2\) children; orient every edge away from the root and add a directed edge from every leaf to all its predecessors; finally weigh every edge \(\weight(u,v) = 1/\dout(u)\) so that the condition \(\wout(v) = 1\) is trivially satisfied for all vertices \(v\). This directed graph has \( n = \Theta \left( k^l \right) = \Theta \left( l^{2l} \right) \) many vertices so that \(\log n = \Theta \left(l \log l\right), \log \log n = \Theta \left( \log l \right)\). One easily checks that the heaviest cycle is any path from the root to a leaf together with the edge from this leaf to the root. This cycle has weight \(l / k + 1/l = O \left( 1/l \right) = O \left( \log \log n / \log n \right)\).

\paragraph{Our results} 

We prove the following theorem making significant progress towards Conjecture~\ref{conj:bollobas_scott}.
\begin{restatable}{theorem}{mainresult}
\label{thm:main_result}
	Every Eulerian digraph $G$ on $n$ vertices with maximum degree $\Delta$ can be decomposed into $O(n\log \Delta)$ edge-disjoint cycles.
\end{restatable}

As a direct consequence of this we can improve the lower bound from Huang, Ma, Shapira, Sudakov and Yuster on the length of a cycle when \(\Omega \left(n \log^2 n \right) \le m \le O\left( n^2 / \log n \right) \). 

\begin{corollary}
	Every Eulerian digraph $G$ on $n$ vertices, $m$ edges and with maximum degree $\Delta$ contains a cycle of length $\Omega\left(\frac{m}{n\log \Delta}\right)$. 
	
\end{corollary}
Our approach connects the problems of finding small cycle decompositions and finding heavy cycles in Eulerian digraphs in the following manner. For any cycle $C$ in a digraph $G$, we define its weight as the sum of inverse out-degrees in $G$ of vertices in $C$. Note that this may equivalently be seen as assigning each arc $(u,v)$ in $G$ the weight $\weight(u,v)=1/\dout(u)$, in which case we clearly have that the sum of out-degrees at every vertex is one. Intuitively, a cycle has a large weight if it is either long or if it consists of small degree vertices. We prove that any digraph with sufficiently large minimum degree contains a cycle of constant weight. Sequentially removing such cycles ensures a decomposition into few cycles. 

The techniques we use can be adapted to general weighted digraphs. They allow us to improve the result of Li and Zhang and show that the upper bound given by Bollob\'as and Scott, and Spencer is asymptotically correct.

\begin{restatable}{theorem}{mainweighted}
	\label{thm:weighted_digraph_heavy_cycle_existence}
	Let $G$ be a digraph with maximum degree $\Delta$ and $\weight \colon E(G)\to \R_+$ be a weight function on the edges such that for every vertex we have $\wout(v)\ge 1$ then there is a cycle $C\subseteq G$ such that \[\weight(C)=\Omega\left(\frac{\log \log \Delta}{\log \Delta}\right).\]
\end{restatable}

\paragraph{Structure of this paper}
This paper is organised as follows. In Section \ref{sec:preliminaries} we introduce general tools and notation. In Section \ref{sec:unweighted_logn_case} we prove slightly weaker versions of Theorems \ref{thm:main_result} and \ref{thm:weighted_digraph_heavy_cycle_existence}. These proofs give the intuition on how we find the cycles using a random walk technique. In Section \ref{sec:weighted_case} we prove Theorem \ref{thm:main_result} and \ref{thm:weighted_digraph_heavy_cycle_existence}. These proofs are more technical and derandomize the techniques used in Section \ref{sec:unweighted_logn_case}.


\section{Preliminaries}
\label{sec:preliminaries}

By $\log$ we denote the natural logarithm. We use standard graph theoretic notation (see e.g.~\cite{D12}). We denote a digraph by $G=(V,E)$ where $V$ is the set of vertices and $E$ is the set of edges (arcs). Given an edge $e=(u,v)$ we say $e$ is an out-edge for $u$ and an in-edge for $v$. For every vertex $v\in V$ we write $\din(v)$ for the in-degree of the vertex $v$, the number of in-edges at the vertex $v$ and $\dout(v)$ for the number of out-edges of $v$. Correspondingly, we write $\Nin(v)$ and $\outneighb(v)$ for the in- and out-neighbourhood of a vertex $v$. If the underlying graph is not clear from the context we add it as a subscript and write e.g.\ $\dout_G(v)$ for clarification. 

For a digraph $G$, we write $\delta(G) \eqdef \min_{v \in V(G)} \{ \din(v), \dout(v) \}$ for its minimum degree and $\Delta(G) \eqdef \max_{v\in V(G)} \{ \din(v), \dout(v) \}$ for its maximal degree. Often, when no confusion can be made, we write $\delta, \Delta$ instead. Furthermore, we write $\bar{\deg}:=|E(G)|/n$ for the average degree of $G$. We call a digraph \emph{sinkless} if \(\dout(v) > 0\) for every vertex \(v\).
We say that a digraph is \emph{balanced} if $\din(v) = \dout(v)$ for every vertex $v$. We make two trivial remarks at this point: an Eulerian digraph is a connected balanced digraph and any digraph obtained by removing a cycle from a balanced digraph is also balanced. 

When talking about weighted graphs, we always consider a weight function $\weight \colon E(G)\to \R_+$. For a vertex $v$, we call the sum of weights of out-edges of $v$ the \emph{out-weight} of that vertex and write it $\wout(v)=\sum_{(v,u)\in E(G)} \weight(v,u)$. For a subgraph $H\subseteq G$ we define the weight of the subgraph to be the sum of all its edges i.e.\ $\weight(H)=\sum_{e\in E(H)}\weight(e)$. We name $\weight_{max} \eqdef \max_{e \in E(G)} \{\weight(e)\}$ the \emph{maximal weight} of the graph.

\begin{definition}\label{def:remdeg}
Given a path $P$, for any vertex $v$ we define the \emph{remaining out-degree} of $v$ as $\drem_P(v) \eqdef \left| \outneighb(v) \setminus V(P) \right|$, 
\ie the number of out-neighbours of $v$ in \(G\) but not in \(P\). When talking about weighted graphs, we also define the \emph{remaining out-weight} of $v$ as $\wrem_P(v) \eqdef \sum_{ u \in \outneighb(v) \setminus V(P) }{ \weight(v,u) }$.
\end{definition}
We often consider the sub-paths $(x_0, \dots, x_t)$ of a path $P = (x_0, \dots, x_T)$; we write $\drem_t$, $\wrem_t$ instead of $\drem_{(x_0, \dots, x_t)}$, $\wrem_{(x_0, \dots, x_t)}$.
\vspace{1em}

 One of the key ingredients of our proof is the following lemma that links existence of heavy cycles to cycle decompositions.
 \begin{lemma}
	\label{lem:small_decomp_existence}
	
	 Let $G$ be a balanced \(n\)-vertex digraph and let $\xi,\mu \geq 0$. If every $G' \subseteq G$ with minimum degree $\delta(G') \geq \mu $ contains a cycle $C \subseteq G'$ such that \[\sum_{v \in C}{ \frac{1}{\dout_{G'}(v)} } \geq \xi,\] then there exists a cycle decomposition of $G$ with at most $ n \mu + (1/\xi)n \log \bar{\deg} + O(n/\xi)$ cycles.
\end{lemma}

\begin{proof}
	We construct a cycle decomposition by sequentially taking out cycles: at each step, we pick one and remove it from the graph until it is empty. Because we start from a balanced graph, at each step the remaining graph is also balanced and we can indeed find a cycle. Formally, let  $(G_t)_{t \geq 0}$, $(C_t)_{t \geq 0}$ be two families of subgraphs of $G$ with $G_0 = G.$
	As long as the graph $G_t$ is non-empty, we will remove a cycle $C_{t}$ in the next round. 
\begin{enumerate}
	\item If $\delta(G_t) < \mu$ then $C_{t}$ is a cycle (chosen arbitrarily) of $G_t$ containing a vertex $v$ of minimum degree.
	\item If $\delta(G_t) \geq \mu$ then $C_{t}$ is a cycle such that $\sum_{v \in C_t}{\frac{1}{\dout_{G_t}(v)}} \geq \xi$. 
\end{enumerate}
	We set 
	\[G_{t+1}=G_t-C_{t}\]
	and remove any isolated vertices.
	Let $T$ be the first index such that $G_T = \emptyset$. This means that $\mathcal{C} =\{C_0, \dots, C_{T-1}\}$ is a cycle decomposition of $G$. We show that $T$ cannot be too large. For this, we look separately at the two types of cycles we find and count how often this can occur. Let $\mathcal{C}_1$ be the collection of all the cycles that were chosen in case \(1\), \ie cycles that contain a minimum degree vertex in the remaining graph. Similarly, let $\mathcal{C}_2$ be the collection of cycles that were chosen in case \(2\). Then clearly $\mathcal{C}=\mathcal{C}_1\cup \mathcal{C}_2$.
	
	The first case occurs if there is a vertex of degree at most $\mu$ in the remaining graph. In this case we can be sure to reduce the degree of this vertex by 1. Note that this case can happen at most $\mu$ times for a particular vertex $v$ being the minimum degree vertex. We conclude that this happens at most $n \mu$ times in total and therefore 
	\[|\mathcal{C}_1|\le n \mu.\]

	 We now upperbound the number of cycles in \(\mathcal{C}_2\). Given how cycles of \(\mathcal{C}_2\) were chosen, we clearly have \[ \xi |\mathcal{C}_2|\le  \sum_{ C_{t} \in \mathcal{C}_2}\hspace{0.5em} \sum_{v \in C_{t}}{ \frac{1}{\dout_{G_t}(v)} } \le \sum_{ C_{t} \in \mathcal{C}}\hspace{0.5em} \sum_{v \in C_{t}}{ \frac{1}{\dout_{G_t}(v)} } . \numberthis\label{eq:cycle_decomp_1}\] 
	Every time a cycle $C \in \mathcal{C}$ is removed, each vertex $v$ in it contributes to the right-hand side of \eqref{eq:cycle_decomp_1} by the inverse of its current out-degree $1/\dout_{G_t}(v)$. 	As we continue the process until there are no edges left we know that the total contribution of a vertex $v$ is exactly $\sum_{i=1}^{\dout_G(v)}1/i$. By a double counting argument we get
	
	\[\sum_{ C_{t}\in \mathcal{C}}\hspace{0.5em}{ \sum_{v \in C_{t}}{ \frac{1}{\dout_{G_t}(v)}}}=\sum_{v \in V(G)}{ \sum_{ i=1 }^{\dout_G(v)}{\frac{1}{i}} }.\numberthis\label{eq:cycle_decomp_2}\]
Combining \eqref{eq:cycle_decomp_1} with \eqref{eq:cycle_decomp_2} and using Jensen inequality for $\log$, we have \begin{align*}
		\xi |\mathcal{C}_2| 
			&\leq \sum_{v \in V(G)}{ \sum_{ i=1 }^{\dout_G(v)}{\frac{1}{i}} } = \sum_{v \in V(G)}{ \left(\log \dout_G(v) + O(1) \right) }
			\leq n \log \bar{\deg} + O(n) .
	\end{align*}
	We conclude 
	\[|\mathcal{C}_2| \leq ({1}/{\xi}) n \log \bar{\deg} + O(n/ \xi),\] and therefore \[T = |\mathcal{C}_1|+|\mathcal{C}_2| \le n \mu + (1/\xi)n \log \bar{\deg} + O(n/\xi).\]
\end{proof}

As a second tool, we have the following lemma, which can be seen as a large deviation result for sums of (not necessarily independent) Bernoulli random variables. This will play a central roll in Section \ref{sec:unweighted_logn_case}.

\begin{lemma} \label{lem:adversarial_probabilities} 
	Let $X_1,\ldots,X_n\in\{0,1\}$ be any stochastic process, and let $p_1,\ldots,p_n \in [0,1]$ be a sequence of random variables such that, for every $t\in [n]$, we have
\begin{equation}\label{eq:wizardcondition}\E[X_t\vert X_1,\ldots,X_{t-1},p_1,\ldots,p_t]=p_t.\end{equation}
	Then, for any $\lambda>0$ and any c we have \[\Pr\left[\lambda \cdot \sum_{t=1}^nX_t>(e^\lambda - 1)\cdot \sum_{t=1}^np_t +c\right]\le e^{-c}.\]
\end{lemma}


We can illustrate the process in the statement of Lemma \ref{lem:adversarial_probabilities} as follows. Suppose you play a game against a magician, which goes on for $n$ rounds. In 
each round $i=1, 2, \dots, n$ you throw a coin, whose outcome $X_i$ 
is either $0$ or $1$. 
Before each throw, the magician uses their magic powers to decide the probability $p_i$ that the coin shows up with $1$. For this, the magician must pay you $p_i$ Swiss Francs, but in return you must pay them one Swiss Franc for every time the coin shows a $1$. The conclusion is, no matter how the magician picks the probabilities, they are unlikely to win much more than they invested. 

The proof of this lemma goes along the lines of the proof of Chernoff's inequality (see e.g.~\cite{FK16}).

\begin{proof}[Proof of Lemma \ref{lem:adversarial_probabilities}]
	Note that we have
	\begin{align*}
	\Pr\left[\lambda \cdot \sum_{t=1}^n X_t>(e^\lambda -1)\cdot \sum_{t=1}^t p_t +c \right]&= \Pr\left[\lambda \cdot \sum_{t=1}^n X_t - (e^\lambda -1)\cdot\sum_{t=1}^n p_t  > c\right]\\
	&=\Pr\left[e^{\left(\lambda \cdot\sum_{t=1}^n X_t - (e^\lambda -1)\cdot\sum_{t=1}^n p_t \right)} > e^{ c}\right].
	\end{align*}
	By Markov's inequality we know that 
	\[\Pr\left[e^{\left(\lambda\cdot\sum_{t=1}^n X_t - (e^\lambda-1)\sum_{t=1}^n p_t \right)} > e^c\right]\le\frac{\E\left[e^{\left(\lambda\cdot\sum_{t=1}^n X_t - (e^\lambda-1)\sum_{t=1}^n p_t \right)}\right]}{e^{ c}}. \]
	In the following we show that \[\E\left[e^{\left(\lambda\cdot\sum_{t=1}^n X_t - (e^\lambda-1)\sum_{t=1}^n p_t \right)}\right]\le 1\] which finishes the proof of the lemma.	We prove this by induction on $n$. 
	Observe that for any $t$ we have
	\begin{align*}
	\E\left[e^{\lambda\cdot X_t-(e^\lambda-1)\cdot p_t}\vert X_1,\ldots,X_{t-1},p_1,\ldots,p_t\right] &= p_t e^{\lambda-(e^\lambda-1)\cdot p_t}+(1-p_t)e^{-(e^\lambda-1)\cdot p_t}\\
	&=e^{-(e^\lambda-1)\cdot p_t}\left(1+p_t(e^\lambda-1)\right)\\
	&\le e^{-(e^\lambda-1)\cdot p_t}\cdot e^{p_t(e^\lambda-1)}\\
	&= 1.
	\end{align*}
	The base case follows directly from setting $t=1$. For the induction step note that
	\begin{align*}
	&\E\left[e^{\left(\lambda\cdot \sum_{t=1}^n X_t - (e^\lambda-1)\cdot \sum_{t=1}^n p_t \right)}\right]\\ &\qquad= \E\left[\E\left[e^{\left(\lambda\cdot \sum_{t=1}^n X_t - (e^\lambda-1)\cdot \sum_{t=1}^n p_t \right)}\vert X_1,\ldots,X_{n-1},p_1,\ldots,p_n\right]\right]\\
	&\qquad=\E\left[e^{\left(\lambda\cdot \sum_{t=1}^{n-1} X_t - (e^\lambda-1)\cdot \sum_{t=1}^{n-1} p_t \right)}\E\left[e^{\left( \lambda\cdot X_n - (e^\lambda-1)\cdot p_n \right)}\vert X_1,\ldots,X_{n-1},p_1,\ldots,p_n\right]\right]\\
	&\qquad\le\E\left[e^{\left(\lambda\cdot \sum_{t=1}^{n-1} X_t - (e^\lambda-1)\cdot \sum_{t=1}^{n-1} p_t \right)}\cdot 1\right] \\
	&\qquad=\E\left[e^{\left(\lambda\cdot \sum_{t=1}^{n-1} X_t - (e^\lambda-1)\cdot \sum_{t=1}^{n-1} p_t \right)}\right]\\
	&\qquad\le 1
	\end{align*}
	where in the last step we use our induction hypothesis for $n-1$.
\end{proof}

\section{A randomized construction}

\label{sec:unweighted_logn_case}

In this section we introduce our strategy and ideas to finding heavy cycles and show how this helps find a cycle decomposition with few cycles. We prove neither Theorem~\ref{thm:main_result} nor Theorem~\ref{thm:weighted_digraph_heavy_cycle_existence} here. Those theorems are proved in Section~\ref{sec:weighted_case}, where we refine our ideas and make them more technical. 
We believe that the core idea is nonetheless simple and adaptable to other problems.

\begin{theorem}
	\label{thm:main_theorem_weak_version}
	Every Eulerian digraph on $n$ vertices can be decomposed into $O(n \log n)$ edge-disjoint cycles.
\end{theorem}

It is easy to see that every Eulerian digraph can be decomposed into edge-disjoint cycles by sequentially taking out cycles. The challenge is to choose these cycles cleverly so that we do not need too many. To prove Theorem~\ref{thm:main_theorem_weak_version}, we proceed as follows: first we add a weighting to the digraph such that every edge $e = (u, v)$ receives weight $\weight(e) = 1/\dout(u)$. We then show that, regardless of how we decompose the graph, the sum of weights of the cycles is bounded; if all cycles in our decomposition have weight sufficiently large, then there cannot be too many of them. The key ingredient is therefore to find heavy cycles. Proposition~\ref{prop:weight_cycle_existence} below shows that this is indeed possible.

\begin{proposition}
	\label{prop:weight_cycle_existence}
	There exist positive constants $K_0$, $K_1$ such that every digraph $G$ of order $n$ with minimum degree $\delta(G) \geq K_0 \cdot \log n$ contains a cycle $C$ such that \begin{align*}
			\sum_{v \in C}{ \frac{1}{\dout(v)} } \geq K_1.
	\end{align*}
\end{proposition}

The proof of Theorem~\ref{thm:main_theorem_weak_version} follows directly from Proposition~\ref{prop:weight_cycle_existence} and Lemma~\ref{lem:small_decomp_existence}.
In the rest of this section we prove Proposition~\ref{prop:weight_cycle_existence}. The central idea of the proof is to consider the following random walk.

\begin{definition}
	\label{def:easy_case_random_walk}
	Given a sinkless digraph $G$, we produce a \emph{random path} $(x_t)_{t \geq 0}$ on $G$ as follows: \begin{itemize}
		\item The first vertex $x_0$ is chosen arbitrarily.
		\item At step $t \geq 0$, if $\drem_t(x_t) \geq \frac{1}{2} \dout(x_t)$, we choose $x_{t+1}$ \uar among the unvisited neighbours of $x_t$.
		\item At step $t \geq 0$, if $\drem_t(x_t) < \frac{1}{2} \dout(x_t)$, we stop the path. We name $T$ the time at which the path stops.
		
	\end{itemize}
\end{definition}

No vertex is visited twice so it is justified that we call it a path. In particular, this implies that $T < n$. To construct a cycle, we connect the last vertex $x_T$ of this path to its first neighbour $x_s$ in the path. In what follows, we show that, for suitable digraphs $G$, the cycle $C = (x_s, \ldots, x_T)$ produced in this manner satisfies the conclusion of Proposition~\ref{prop:weight_cycle_existence} w.h.p.\footnote{We say that a sequence of events $E_1, E_2, \dots$ holds \emph{with high probability} (or w.h.p.\ for short) if $\Pr\left[ E_n\right] \rightarrow 1$ as $n\rightarrow\infty$.} For this, we first prove that the visited out-neighbourhood of any vertex between two steps cannot be too large w.h.p.

\begin{lemma}
	\label{lem:visited_neighbours_bound}

	For a sinkless digraph $G$, let $(x_0, \ldots, x_T)$ be a random path and let $\lambda>0$. With probability at least $1-1/n$, it holds for all $v \in G$ and all $s < T$ that \begin{align*}
		 \left| \left\{ \outneighb(v) \cap \{x_s, \dots, x_T\} \right\} \right| \leq 1+\frac{2e^\lambda-2}{\lambda} \cdot \dout(v) \cdot \sum_{t=s}^{T-1}{\frac{1}{\dout(x_t)}} + \frac{3}\lambda \cdot \log n.
	\end{align*}
\end{lemma}
\begin{proof}

	Before formally proving it, we give an intuition of why this is true. Let $v$ be a fixed vertex in $G$, and consider the number of out-neighbours of $v$ that are contained in the path. If at time $t$ the random path has visited $x_0, \dots, x_{t}$ and does not stop yet, \ie $\drem_t(x_t) \geq \frac{1}{2} \dout(x_t)$, then we have\[\Pr[ x_{t+1}\in \outneighb( v)] = \frac{\left| \outneighb(v) \cap \outneighb(x_t) \setminus \{ x_0, \dots, x_{t} \} \right|}{\left| \outneighb(x_t) \setminus \{ x_0, \dots, x_t \} \right|} \leq 2\frac{\dout(v)}{\dout(x_t)}.\] Because we select a random path, we expect the number of visited vertices in the out-neighbourhood of $v$ to be concentrated around its expectation, that is, not too different from $\sum_t{ {\dout(v)}/{\dout(x_t)}}$. Thus, for a typical vertex $v$ we would expect the random path to enter its out-neighbourhood not much more than $2\dout(v)\sum_{t=s}^{T-1} {1}/{\dout(x_t)}$ times after time $s$.
	
Formally, for a given $v \in G$ and for $1 \leq t \leq n$ we define the random variables \[X^{(v)}_t \eqdef \indicator{t \leq T \land x_t \in \outneighb(v)}\text{ and }p^{(v)}_t \eqdef \indicator{t \leq T} \cdot \Pr \left[ x_t \in \outneighb(v) \mid x_0, \dots, x_{t-1} \right].\]

In what follows we wish to apply Lemma~\ref{lem:adversarial_probabilities} to these quantities. Hence, we need to check that the conditions of the lemma, in particular \eqref{eq:wizardcondition}, are satisfied. To do so, we observe that, for any $t$, we have \begin{align*}
		&\E \left[ X^{(v)}_t \mid X^{(v)}_1, \dots, X^{(v)}_{t-1}, p^{(v)}_1, \dots, p^{(v)}_t \right] \\
			&\quad = \E \biggl[ \E \bigl[ X^{(v)}_t \bigm| X^{(v)}_1, \dots, X^{(v)}_{t-1}, p^{(v)}_1, \dots, p^{(v)}_t, x_0, \dots, x_{t-1} \bigr] \biggm| X^{(v)}_1, \dots, X^{(v)}_{t-1}, p^{(v)}_1, \dots, p^{(v)}_t \biggr] \\
			&\quad = \E \biggl[ \E \bigl[ X^{(v)}_t \mid x_0, \dots, x_{t-1} \bigr] \biggm| X^{(v)}_1, \dots, X^{(v)}_{t-1}, p^{(v)}_1, \dots, p^{(v)}_t \biggr] \\
			&\quad = \E \biggl[ p^{(v)}_t \biggm| X^{(v)}_1, \dots, X^{(v)}_{t-1}, p^{(v)}_1, \dots, p^{(v)}_t \biggr] \\
			&\quad = p^{(v)}_t,
	\end{align*}
	as desired.
	
To apply Lemma~\ref{lem:adversarial_probabilities}, we first note that, for any vertex $v$ and any $s \geq 0$, we have on the one hand that \begin{equation}\label{eq:visited_neighbours_bound_proof_2}
\left| \outneighb(v) \cap \{x_s, \dots, x_T\} \right| \leq 1 + \left| \outneighb(v) \cap \{x_{s+1}, \dots, x_T\} \right| = 1+\sum_{t=s+1}^{n}{X^{(v)}_t} 
	\end{equation} where we interpret $\{x_s, \dots, x_T\}$ as the empty set when $s > T$. On the other hand, as we noted above,
		$p^{(v)}_t \leq 2{\dout(v)}/{\dout(x_{t-1})}$
for any $1\leq t \leq T$ and, by definition, $p^{(v)}_t=0$ for $t>T$. It follows that, for any $s\geq 0$ we have 
	\begin{equation}\label{eq:visited_neighbours_bound_proof_1}
		\sum_{t=s+1}^{n}{ p_t^{(v)} } \leq 2 \sum_{t=s}^{T-1}{ \frac{\dout(v)}{\dout(x_t)} }.
	\end{equation}
By Lemma~\ref{lem:adversarial_probabilities}, we have for any vertex $v$ and any $s \geq 0$ that \begin{align*}
		\Pr \left[ \sum_{t=s+1}^n{ X^{(v)}_t } > \frac{e^\lambda-1}\lambda \sum_{t=s+1}^n{p^{(v)}_t} + \frac{3}{\lambda} \log n \right] \leq n^{-3}.
	\end{align*}
By combining \eqref{eq:visited_neighbours_bound_proof_2} and \eqref{eq:visited_neighbours_bound_proof_1}, it follows that
	\begin{align*}
 		\Pr \left[ \left| \outneighb(v) \cap \{x_{s}, \dots, x_T\} \right| > 1 + \frac{2e^\lambda-2}\lambda \dout(v) \sum_{t=s}^{T-1}{\frac{1}{\dout(x_t)}} + \frac{3}{\lambda} \log n \right] \leq n^{-3}.
 	\end{align*}
 	By union bound over all $v$, $s$, we obtain the claim.
\end{proof}

Recall that Proposition~\ref{prop:weight_cycle_existence} states that provided a digraph has minimum degree large enough, we can find a heavy enough cycle. We are now well equipped to prove this lemma.

\begin{proof}[Proof of Proposition~\ref{prop:weight_cycle_existence}]
	Note that $\delta(G)>0$ ensures that $G$ is sinkless.
	In this proof we use Lemma~\ref{lem:visited_neighbours_bound} to show that with positive probability, the weight of the created cycle is high enough.
		
	Let $G$ be a graph of order $n$ with $\delta \geq K_0 \cdot \log n$. Let $(x_0, \ldots, x_T)$ be a random path and let $x_s$ be the first neighbour of $x_T$ in the path so that $(x_s, \ldots, x_T)$ is a cycle. By definition, $x_T$ has at least $\dout(x_T) / 2$ out-neighbours in the path, $x_s$ being the first one. Hence \begin{align}
		\frac{1}{2} \dout(x_T) \leq \left| \outneighb(x_T) \cap \{x_s, \dots, x_T\} \right| \label{eq:weight_cycle_existence_proof_2}.
 	\end{align}
Moreover, applying Lemma~\ref{lem:visited_neighbours_bound} with $\lambda=1$, $v=x_T$ and $s$ as above, we have, with positive probability, that \begin{equation}\label{eq:weight_cycle_existence_proof_3}\begin{split}
		\left| \outneighb(x_T) \cap \{x_s, \dots, x_T\} \right| 
			&\leq 1 + (2e-2) \cdot \dout(x_T) \cdot \sum_{t=s}^{T-1}{ \frac{1}{\dout(x_t)}} + 3 \cdot \log n  \\
			&\leq 4 \cdot \dout(x_T) \cdot \sum_{t=s}^{T-1}{ \frac{1}{\dout(x_t)}} + 4 \log n.  \end{split}\end{equation}
	Combining \eqref{eq:weight_cycle_existence_proof_2}, \eqref{eq:weight_cycle_existence_proof_3} we obtain
	\begin{align*}
		\frac{1}{2} \dout(x_T) \leq 4 \dout(x_T) \cdot \sum_{t=s}^{T-1}{ \frac{1}{\dout(x_t)} } + 4 \cdot \log n.
	\end{align*}
	Rearranging the terms gives \begin{align*}
		\sum_{t=s}^T{ \frac{1}{\dout(x_t)} } \geq \frac{1}{8} - \frac{\log n}{\dout(x_T)} \geq \frac{1}{8} - \frac{1}{K_0},
	\end{align*} since $\dout(x_T) \geq \delta(G) \geq K_0 \log n$. The lemma follows by choosing $K_0>8$ and $K_1:=1/8-1/K_0.$
\end{proof}

To conclude this section, we show that it is possible to use the random path in Definition~\ref{def:easy_case_random_walk} to prove a weaker version of Theorem~\ref{thm:weighted_digraph_heavy_cycle_existence} in the case of uniform out-weights. Extending this to a full proof of the theorem requires some additional ideas, which will be discussed in the next section.

\begin{proposition}
Let $G$ be a sinkless digraph on $n$ vertices. Consider a random path $(x_0, \dots, x_T)$ on $G$ and let $C$ be the corresponding cycle. Then, with high probability as $n\rightarrow\infty$ we have
\begin{equation*}\label{eq:rwgengraph}\sum_{v\in C} \frac{1}{\dout(v)} \geq \frac{\log\log n}{8 \log n}.\end{equation*}
\end{proposition}
\begin{proof}
 Let \(s\) be the first index such that \(x_s \in \outneighb(x_T)\). Applying Lemma~\ref{lem:visited_neighbours_bound} with $\lambda=\log \log n$ it follows that, with high probability,
\begin{equation}
\label{eq:weight_cycle_existence_proof_4}\begin{split}
		\left| \outneighb(x_T) \cap \{x_s, \dots, x_T\} \right| 
			&\leq 1 + \frac{2\log n-2}{\log \log n} \cdot \dout(x_T) \cdot \sum_{t=s}^{T-1}{ \frac{1}{\dout(x_t)}} + 3 \cdot \frac{\log n}{\log \log n}  \\
			&\leq \frac{2 \log n}{\log \log n} \cdot \dout(x_T) \cdot \sum_{t=s}^{T-1}{ \frac{1}{\dout(x_t)}} + 4 \frac{\log n}{\log \log n}.  \end{split}
\end{equation}
By Definition~\ref{def:easy_case_random_walk}, the random path ends when \(\drem_T(x_T) < \dout(x_T) / 2 \), which implies that \[ \dout(x_T) / 2 < \dout(x_T)- \drem_T(x_T) = \left| \outneighb(x_T) \cap \{x_s, \ldots, x_T\} \right|.\]By combining this together with \eqref{eq:weight_cycle_existence_proof_4} and rearranging the terms, it follows that
\begin{equation*}
\sum_{t=s}^{T-1} \frac{1}{\dout(x_t)} \geq \frac{1}{4}\frac{\log \log n}{\log n} - \frac{2}{\dout(x_T)}.
\end{equation*}

Recall that the cycle corresponding to the random path was described as \(C = (x_s, \ldots, x_T)\), obtained by taking the edge from the last vertex \(x_T\) to its first neighbour \(x_s\) in the path. We can conclude that 
\begin{equation*}
\sum_{v\in C} \frac{1}{\dout(v)} = \frac{1}{\dout(x_T)} + \sum_{t=s}^{T-1} \frac{1}{\dout(x_t)} \geq \max\left\{\frac{1}{\dout(x_T)}, \frac{1}{4}\frac{\log \log n}{\log n} - \frac{1}{\dout(x_T)}\right\},
\end{equation*}
and the proposition follows by minimizing the right-hand side over $\dout(x_T)$.
\end{proof}


\section{Derandomization}
\label{sec:weighted_case}

In this section we extend the ideas of our walk of Section~\ref{sec:unweighted_logn_case} to weighted graphs. We provide Theorem~\ref{thm:weighted_general_weight_cycle_high} below, from which we derive Theorems~\ref{thm:main_result} and~\ref{thm:weighted_digraph_heavy_cycle_existence}. 

\begin{theorem}
	\label{thm:weighted_general_weight_cycle_high}
	Let $G$ be a weighted digraph with maximum degree $\Delta \geq 20$ and let $\lambda \in [1, \log \log \Delta]$. If $G$ has maximum weight $\weight_{max} \leq {\lambda}/{50 \log \Delta}$ and is such that every vertex $v$ has out-weight $\wout(v) = 1$, then $G$ contains a cycle $C$ of weight at least \begin{align*}
		\weight(C) \geq \frac{\lambda}{50 e^\lambda}.
	\end{align*}
\end{theorem}

Note that the condition $\Delta \geq 20$ ensures that $\log \log \Delta \geq 1$. We can already note two inequalities that will help us throughout our analysis. First, since $\lambda \leq \log \log \Delta$ we have $\weight_{max} \leq {\lambda}/{50 \log \Delta} \leq {\lambda}/{50 e^\lambda}$. Second, simple analysis shows that $x \mapsto \frac{x}{e^x}$ is maximum in $x = 1$ so \begin{align*}
		\weight_{max} \leq \frac{\lambda}{50 e^\lambda} \leq \frac{1}{100}. \numberthis \label{eq:w_max_simple_bound_1}
\end{align*}
Before proving Theorem~\ref{thm:weighted_general_weight_cycle_high} we show how it implies Theorems~\ref{thm:main_result} and~\ref{thm:weighted_digraph_heavy_cycle_existence}. In the case of Theorem~\ref{thm:main_result} it is useful to note the following consequence of Theorem~\ref{thm:weighted_general_weight_cycle_high}.

\begin{corollary}
	\label{cor:weighted_general_constant_weight_cycle_existence}
	There exists a constant \(\alpha > 0\) such that the following holds. Let $G$ be a digraph with maximum degree $\Delta$ and minimum degree $\delta \geq 50 \log \Delta > 0$. Then $G$ contains a cycle $C$ such that $\sum_{v \in C}{ \frac{1}{\dout(v)} } \geq \alpha$.
\end{corollary}

\begin{proof}
	Choose \(\alpha = 1 / 50e\). Note that $G$ is sinkless because of the minimum degree condition, so it contains a cycle. If $\Delta \leq 20$ then any cycle $C$ is such that $\sum_{v \in C}{1 / \dout(v)} \geq \sum_{v \in C}{1 / \Delta} \geq 1/20 \ge \alpha$. If not, we give every edge $e = (u,v)$ the weight $\weight(e) = {1}/{\dout(u)}$. This ensures that $\wout(v) = 1$ for all $v$ and that $\weight_{max} = {1}/{\delta} \leq {1}/{(50 \log \Delta)}$. It then suffices to apply Theorem~\ref{thm:weighted_general_weight_cycle_high} with $\lambda = 1$. 
\end{proof}
 Note that Theorem~\ref{thm:main_result} follows directly from combining Corollary~\ref{cor:weighted_general_constant_weight_cycle_existence} with Lemma~\ref{lem:small_decomp_existence}.

\mainweighted*
\begin{proof}[Proof of Theorem~\ref{thm:weighted_digraph_heavy_cycle_existence}]
	Consider a digraph $G$ such that $\wout(v) \geq 1$ for all $v \in V$. By possibly decreasing the weight of some edges we may assume without loss of generality that $\wout(v) = 1$ for all vertices $v$. Note that this condition ensures that $G$ is sinkless; let $G' \subseteq G$ be a strongly connected component with no out-edge. Let $\Delta' = \Delta(G') \leq \Delta$ the maximum degree of $G'$. 
	
	Since $G'$ is strongly connected, we can extend any edge to a cycle. If $\Delta' < 20$, then there is an edge, and therefore by extension, a cycle with weight at least $1/20 \geq \Omega( \log \log \Delta /\log \Delta )$. If $\Delta' \geq 20$ and there exists an edge $e \in G'$ with weight $\weight(e) > \log \log \Delta / 50 \log \Delta$ then we can extend this edge to a cycle which has at least that weight, that is of weight at least $\Omega(\log \log \Delta / \log \Delta)$. Finally, if $\Delta' \geq 20$ and every edge in $G'$ has weight at most $\weight_{max} \leq \log \log \Delta / 50 \log \Delta \leq \log \log \Delta' / 50 \log \Delta'$, then by Theorem~\ref{thm:weighted_general_weight_cycle_high} there exists a cycle $C$ of weight at least $\weight(C) \geq \log \log \Delta' / 50 \log \Delta' \geq \Omega( \log \log \Delta / \log \Delta)$
\end{proof}

In the remaining part of this section we prove Theorem~\ref{thm:weighted_general_weight_cycle_high}. Given the intuition from the last section, instead of using randomness we cleverly choose the next vertex to ensure the weight of our cycle is sufficiently large. We start with some definitions. 

\begin{definition}
	\label{def:active2}
	Given a directed graph $G$ and a path $P = (v_0, \dots, v_t)$ in $G$, we say that a vertex $v$ is activated if $v$ has an out-neighbour in $P$. We define the activation time $v$ as $act_P(v) \eqdef \min \{s \mid v_s \in \outneighb(v) \}$.
\end{definition}

For the next definition, recall that in Definition~\ref{def:remdeg} we defined the remaining degree with respect to a fixed sub-path. In particular, recall that $\drem_r(v)$ does not depend on the whole path but only on $v_0,\ldots,v_r$ and thus does not change when we extend the path.

\begin{definition}
	\label{def:annoyance2}
	Given a directed graph $G$ and a path $P = (v_0, \dots, v_t)$ of length $t$ in $G$ and an activated vertex $v$ we define \begin{align*}
		a_{P}(v) \eqdef e^{ \frac{1}{\weight_{max}} \sum_{r = act_P(v)}^{t-1}{ \left( \lambda \cdot \weight(v,v_{r+1}) - e^\lambda \cdot \frac{1}{\drem_r(v_r)} \right) } }.
	\end{align*}
	We also define $A_{P} \eqdef \frac{1}{\Delta^2} \sum_{ v \text{ activated} }{ a_{P}(v) }$.
\end{definition}

\begin{definition}
	\label{def:bnnoyance}
	Given a directed graph $G$ and a path $P = (v_0, \dots, v_t)$ of length $t$ in $G$, for any $s < t$ we define \begin{align*}
		b_{P}(s) \eqdef e^{ -\frac{50 e^\lambda}{\lambda} \sum_{r=s}^{t-1}{ \weight(v_r, v_{r+1}) } }.
	\end{align*} 
	We also define $B_{P} \eqdef \sum_{s=0}^{t-1}{ \frac{1}{\drem _s(v_s)} \cdot b_{P}(s) }$.
\end{definition}

Roughly speaking, the quantity \(A\) controls the \emph{sum of inverse of remaining degrees} along the path and allows us to obtain bounds similar to those obtained through Lemma~\ref{lem:adversarial_probabilities} in Section~\ref{sec:unweighted_logn_case}. If we can guarantee that \(A\) is sufficiently small, then, as we show below, we can also guarantee the existence of a cycle with high \emph{sum of inverse of remaining degrees}. The quantity $B$ is designed to make sure that, as long as it is low enough, the \emph{sum of inverse of remaining degrees} on any interval of the path is comparable to the \emph{weight} of the path on this interval. We want to keep both $A$ and $B$ low enough so that, combining both, we can find a cycle which is heavy enough.

\begin{proposition}
	\label{prop:weighted_general_ab_bounded}
	Let $G$ be a weighted digraph with maximum degree $\Delta \geq 20$ and let $\lambda \in [1, \log \log \Delta]$. If $G$ has maximum weight $\weight_{max} \leq {\lambda}/{50 \log \Delta}$ and is such that every vertex $v$ has out-weight $\wout(v) = 1$, then there exists a path $\Pi = (x_0, \dots, x_T)$ in $G$ such that $\wrem_{\Pi}(x_T) < 1/2$ and \begin{align*}
		A_\Pi + B_\Pi \leq \frac{7 \lambda}{50 e^\lambda}.
	\end{align*}
\end{proposition}

The path $\Pi$ of Proposition~\ref{prop:weighted_general_ab_bounded} can be seen as the output of a (deterministic) algorithm that creates a path by walking on the graph, greedily minimising $A + B$ at each step. To prove this proposition, we show that at each step of this walk, the quantity $A + B$ cannot grow too much. For that, we use Claims~\ref{lem:weighted_general_a_expectation} and~\ref{lem:weighted_general_b_expectation} below, which state that $A$, $B$ would not grow too much in expectation if we decided to choose the next vertices randomly. We assume the conditions on $G$, $M$ and $\lambda$ from Proposition~\ref{prop:weighted_general_ab_bounded} hold.

\begin{claim}
	\label{lem:weighted_general_a_expectation}
	Let $P = (v_0, \dots, v_t)$ be \emph{any} path in $G$ such that $\wrem_t(v_t) \geq {1}/{2}$. Consider the path $P'=(v_0, \dots, v_t, u)$ obtained by adding $u$ chosen u.a.r.\ among the unvisited neighbours of $v_t$. We have \begin{align*}
		\E \left[ A_{P'} \right] \leq \left( 1 - 15 \frac{e^\lambda}{\lambda \cdot \drem_t(v_t)} \right) \cdot A_{P} + \frac{1}{\Delta}.
	\end{align*}
\end{claim}

\begin{proof}
	By definition of $a$, for any vertex $v$ activated by $P$ and $u \in \outneighb(v_t) \setminus P$, 
	\begin{align}
		a_{P'}(v) = a_{P}(v) \cdot e^{ \frac{\lambda}{\weight_{max}} \weight(v,u) - \frac{e^\lambda}{\weight_{max}} \frac{1}{\drem_t(v_t)} }.\label{eq:weighted_general_a_expectation_proof_1}
	\end{align}
	Note that $0 \leq \weight(v,u) \leq \weight_{max}$ and since $x \mapsto e^{ \frac{\lambda}{\weight_{max}} x }$ is convex, $e^{\frac{\lambda}{\weight_{max}} \weight(v,u) } \leq 1 + \frac{e^{\lambda} - 1}{\weight_{max}} \weight(v,u)$. 
	Also, observe that, since we choose $u$ \uar among the unvisited neighbours of $v_t$, \begin{align*}
 			\E \left[ \weight(v,u) \right] 
 				&= \sum_{u' \in \outneighb(v_t) \setminus V(P)}{ \left( \Pr[u = u'] \cdot \weight(v,u') \right) } \\
 				&= \frac{1}{\drem_t(v_t)} \cdot \sum_{u' \in \outneighb(v_t) \setminus V(P)}{ \weight(v,u') } \\
 				&\leq \frac{1}{\drem_t(v_t)},
	\end{align*} 
	where the inequality comes from the fact that the sum of out-weights is at most \(\wout(v) = 1\). 
	Hence in expectation \begin{align}
		\E \left[ e^{ \frac{\lambda}{\weight_{max}} \weight(v,u)} \right] 
			\leq 1 + \frac{e^\lambda - 1}{\weight_{max}} \cdot \frac{1}{\drem_t(v_t)} 
			\leq e^{ \frac{e^\lambda - 1}{\weight_{max}} \frac{1}{\drem_t(v_t)} }. \label{eq:weighted_general_a_expectation_proof_2}
	\end{align}
	Combining \eqref{eq:weighted_general_a_expectation_proof_1} and \eqref{eq:weighted_general_a_expectation_proof_2} gives \begin{align*}
		\E \left[ a_{P'}(v) \right]
			&= a_{P}(v) \cdot e^{- \frac{e^\lambda}{\weight_{max}} \frac{1}{\drem _t(v_t)}} \cdot \E \left[ e^{ \frac{\lambda}{\weight_{max}}\weight(v,u) } \right] \\
			&\leq a_{P}(v) \cdot e^{- \frac{e^\lambda}{\weight_{max}} \frac{1}{\drem_t(v_t)}} \cdot e^{ \frac{e^\lambda - 1}{\weight_{max}} \frac{1}{\drem_t(v_t)} } \\
			&\leq a_{P}(v) \cdot e^{- \frac{1}{\weight_{max} \cdot \drem _t(v_t)}}.
	\end{align*}
	We have $\wrem_t(v_t) \geq {1}/{2}$ and we know that every edge has weight at most $\weight_{max}$; this implies $\drem_t(v_t) \cdot \weight_{max} \geq \wrem_t(v_t) \geq {1}/{2}$ or, in other words, $0 \leq {1}/{(\weight_{max} \cdot \drem_t(v_t))} \leq 2$. Since $x \mapsto e^{-x}$ is convex we have \begin{align*}
		e^{- \frac{1}{\weight_{max} \cdot \drem_t(v_t)}} \leq 1 + \frac{e^{-2} - 1}{2} \frac{1}{\weight_{max} \cdot \drem_t(v_t)} \leq 1 - 0.3 \frac{1}{\weight_{max} \cdot \drem_t(v_t)}.
	\end{align*}
	We noted in \eqref{eq:w_max_simple_bound_1} that $\weight_{max} \leq {\lambda}/{50 e^\lambda}$ so in expectation, the contribution of any vertex $v$ that was already activated by $P$ is going to decrease \[\E \left[ a_{P'}(v) \right] \leq \left( 1 - 15 \frac{e^\lambda}{\lambda \cdot \drem_t(v_t)} \right) \cdot a_P(v).\]
	
	To find an upper bound on the expectation of the total $A_{P'}$ we note that any new activated vertex contributes to it by $1/\Delta^2$ and that only the in-neighbours of $u$ can be `newly' activated. $u$ has at most $\Delta$ in-neighbours so the total contribution of newly activated vertices is at most $1 / \Delta$. We therefore have \begin{align*}
		\E \left[ A_{P' } \right]
			&\le \left( 1 - 15 \frac{e^\lambda}{\lambda \cdot \drem_t(v_t)} \right) \cdot A_P + \frac{1}{\Delta}.
	\end{align*}
	
\end{proof}

\begin{claim}
	\label{lem:weighted_general_b_expectation}
	Let $P = (v_0, \dots, v_t)$ be \emph{any} path in $G$ such that $\wrem_t(v_t) \geq {1}/{2}$. Consider the path $P'=(v_0, \dots, v_t, u)$ obtained by adding $u$ chosen u.a.r.\ among the unvisited neighbours of $v_t$. We have \begin{align*}
	\E \left[ B_{P'} \right] \leq \left( 1 - 15 \frac{e^\lambda}{\lambda \cdot \drem_t(v_t)} \right) \cdot B_{P } + \frac{1}{\drem_t(v_t)}.
	\end{align*}
\end{claim}
\begin{proof}
	By definition of $b$, for any $s < t$ we have 
 \begin{align}
		b_{P'}(s)
			&= b_{P}(s) \cdot e^{- \frac{50 e^\lambda}{\lambda} \weight(v_t,u)}. \label{eq:weighted_general_b_expectation_proof_1}
	\end{align}
	Recall that in \eqref{eq:w_max_simple_bound_1} we explained that $\weight_{max} \leq {\lambda}/{50 e^\lambda}$. Thus, $0 \leq \weight(v_t,u) \leq {\lambda}/{50 e^\lambda}$ and since $x \mapsto e^{- \frac{50 e^\lambda}{\lambda} x }$ is convex, \begin{align*}
		e^{ -\frac{50 e^\lambda}{\lambda} \weight(v_t,u) } 
			\leq 1 + \frac{e^{-1} - 1}{\frac{\lambda}{50 e^\lambda}} \cdot \weight(v_t,u)
			\leq 1 - 30 \frac{e^\lambda}{\lambda} \cdot \weight(v_t, u). 
	\end{align*}
	Note that \[\E \left[ \weight(v_t, u) \right] = \frac{\wrem_t(v_t)}{\drem_t(v_t)} \geq \frac{1}{2 \cdot \drem_t(v_t)},\] since we know $\wrem_t(v_t) \geq {1}/{2}$. Hence in expectation \begin{align}
		\E \left[ e^{ -\frac{50 e^\lambda}{\lambda} \weight(v_t, u) } \right]
			\leq 1 - 15 \frac{e^\lambda}{\lambda \cdot \drem_t(v_t) }. \label{eq:weighted_general_b_expectation_proof_3}
	\end{align} 
	Combining \eqref{eq:weighted_general_b_expectation_proof_1} and \eqref{eq:weighted_general_b_expectation_proof_3} gives
	\begin{align*}
		\E \left[ b_{P'}(s) \right] \leq b_{P}(s) \cdot \left( 1 - 15 \frac{e^\lambda}{\lambda \cdot \drem_t(v_t) } \right).
	\end{align*}
	For any time $s<t$, the contribution of $b(s)$ to $B$ decreases in expectation by a factor $\left( 1 - 15 \frac{e^\lambda}{\lambda \cdot \drem_t(v_t) } \right)$. The contribution $b(t)$ is at most $1$ and is weighted by a factor ${1}/{\drem_t(v_t)}$ in $B$. Hence, in expectation \begin{align*}
		\E \left[ B_{P'} \right]
			&\leq \left( 1 - 15 \frac{e^\lambda}{\lambda \cdot \drem_t(v_t) } \right) \cdot B_{P} + \frac{1}{\drem_t(v_t)}.
	\end{align*}
\end{proof}

With Claims~\ref{lem:weighted_general_a_expectation} and~\ref{lem:weighted_general_b_expectation} proven we now show that Proposition~\ref{prop:weighted_general_ab_bounded} holds.

\begin{proof}[Proof of Proposition~\ref{prop:weighted_general_ab_bounded}]
	Consider the following (deterministic) walk on $G$: \begin{itemize}
		\item The first vertex $x_0$ is chosen arbitrarily.
		\item At each step $t$ such that $\wrem_t(x_t) \geq {1}/{2}$, the next vertex of the walk is $$x_{t+1} = \argmin_{u \in  \outneighb(x_t) \setminus \{x_0, \dots, x_t\} } \left( A_{(x_0, \dots, x_{t}, u)} + B_{(x_0, \dots, x_{t}, u)} \right).$$
		\item At step $t$ such that $\wrem_t(x_t) < {1}/{2}$, the walk stops. We call $T:=t$ the stopping time of the walk.
	\end{itemize}
	Note that this walk never self-intersects; we name $\Pi = (x_0, \dots, x_T)$ the path thus produced. For any $0 \leq t \leq T$, we also define $\Pi_t \eqdef (x_0, \dots, x_t)$ the path at step $t$, made of the first $t+1$ vertices of $\Pi$. We prove by induction that, for any $t$, \[A_{\Pi_t} + B_{\Pi_t} \leq \frac{7 \lambda}{50 e^\lambda}.\numberthis\label{eq:ab_bound_IH}\] 
	
	 Recall that we defined \(B_P\) as a sum of \(l\) terms where \(l\) is the length of the path \(P\). Because \(\Pi_0\) has length \(l = 0\), we have \(B_{\Pi_0} = 0\). For a path \(P\) of length \(l\) and an activated vertex \(v\), \(a_P(v)\) was defined as the exponential of a sum of \(l\) terms. Again, \(\Pi_0\) has length \(0\) so \(a_{\Pi_0}(v) = 1\) for every activated vertex \(v\). Because \(G\) has maximum degree \(\Delta\), at most \(\Delta\) vertices are activated by \(\Pi_0\). Therefore \[A_{\Pi_0} = \frac{1}{\Delta^2} \sum_{\text{\(v\) activated}}{a_{\Pi_0}(v)} \le \frac{1}{\Delta} \]
	 and we deduce that $A_{\Pi_0} + B_{\Pi_0} \leq 1/\Delta$. Observe that \[x \mapsto \frac{x}{e^x}\] is decreasing on $[1, \log \log \Delta]$, so if we can prove that $1 / \Delta \leq 7 \log \log \Delta / 50 \log \Delta $ --- or, equivalently that $\Delta \geq 50 \log \Delta / 7 \log \log \Delta $ --- the induction hypothesis holds for $t=0$, for every $\lambda \in [1, \log \log \Delta]$. A bit of analysis shows that \[x \mapsto \frac{50 \log x}{7 \log \log x} - x\] is decreasing on $[20, +\infty)$ and negative in $x = 20$, which proves that \eqref{eq:ab_bound_IH} is true for $t = 0$.
	
	 Suppose \eqref{eq:ab_bound_IH} is true up to some $t < T$; recall that this implies $\wrem_t(x_t) \geq {1}/{2}$. If we prolong $\Pi_t = (x_1, \dots, x_t)$ to $\Pi_t' = (x_1, \dots, x_t, u)$ by picking the next vertex $u$ \uar in the unvisited neighbourhood of $x_t$, then by linearity of expectation and Claims~\ref{lem:weighted_general_a_expectation} and~\ref{lem:weighted_general_b_expectation} we have \begin{align*}
		\E \left[ A_{\Pi_t'} + B_{ \Pi_t' } \right]
			&\leq \left( 1 - 15 \frac{e^\lambda}{\lambda \cdot \drem_t(x_t) } \right) \left( A_{\Pi_t} + B_{\Pi_t} \right) + \frac{1}{\Delta} + \frac{1}{\drem _t(x_t)} \\
			&\leq \left( 1 - 15 \frac{e^\lambda}{\lambda \cdot \drem_t(x_t) } \right) \left( A_{\Pi_t} + B_{\Pi_t} \right) + \frac{2}{\drem _t(x_t)}.
	\end{align*} 
	Since, by definition, $x_{t+1}$ is the neighbour of $x_t$ chosen to minimise the quantity $A + B$, we know that it is at most its expectation. Using \eqref{eq:ab_bound_IH} we get \begin{align*}
		 A_{\Pi_{t+1}} + B_{\Pi_{t+1}} 
			&\leq \left( 1 - 15 \frac{e^\lambda}{\lambda \cdot \drem_t(x_t) } \right) \left( A_{\Pi_t} + B_{\Pi_t} \right) + \frac{2}{\drem _t(x_t)}\\
			&\leq \left( 1 - 15 \frac{e^\lambda}{\lambda \cdot \drem_t(x_t) } \right) \cdot \frac{7 \lambda}{50 e^\lambda} + \frac{2}{\drem _t(x_t)} \\
			&\leq \frac{7 \lambda}{50 e^\lambda}.
	\end{align*} 
	Therefore, \eqref{eq:ab_bound_IH} is true for all $t \leq T$, and in particular, for $t = T$.
	
\end{proof}

We have established with Proposition~\ref{prop:weighted_general_ab_bounded} that there exists a path on which the quantities $A$, $B$ remained bounded along the whole path. We now turn our attention to the weight and sum of inverse of remaining degrees in the cycle and prove that they are sufficiently large.

\begin{claim}
	\label{lem:weighted_general_SIRD_cycle_high}
	Let $\Pi = (x_0, \dots, x_T)$ be  a path with $\wrem_{\Pi}(x_T) < 1/2$ and \[ 
	A_\Pi + B_\Pi \leq \frac{7 \lambda}{50 e^\lambda},\] and let $s = act_\Pi(x_T)$ be the activation time of the last vertex, then \begin{align*}
		\sum_{t=s}^{T-1}{ \frac{1}{\drem _t(x_t)} } \geq \frac{2\lambda}{5 e^\lambda}.
	\end{align*}
\end{claim}

\begin{proof}
	We know that $A_\Pi + B_\Pi \leq \frac{7 \lambda}{50 e^\lambda} \leq 1$ since we noted in \eqref{eq:w_max_simple_bound_1} that ${\lambda}/{50 e^\lambda} \leq {1}/{100} $. This implies \begin{align*}
		a_{\Pi}(x_T)
			\leq \Delta^2 \cdot A_\Pi
			\leq \Delta^2.
	\end{align*}
	Furthermore, since $\wrem_T(x_T) \leq {1}/{2}$, we have $\sum_{t=s}^T{\weight(x_T, x_t)} = \wout(x_T) - \wrem_T(x_T) \geq {1}/{2}$. Since every edge has weight at most $\weight_{max} \leq {1}/{100}$, we get $\sum_{t=s}^{T-1}{\weight(x_T, x_{t+1})} \geq {1}/{2} - {1}/{100}$. Recall that $a_{\Pi}(x_T)$ is defined as \begin{align*}
		a_{\Pi}(x_T) = e^{ \frac{1}{\weight_{max}} \sum_{t=s}^{T-1}{ \left( \lambda \cdot \weight(x_T, x_{t+1}) - e^\lambda \cdot \frac{1}{\drem _t(x_t)} \right) } },
	\end{align*} so, by monotonicity of $\log$, we obtain \begin{align*}
		2 \log \Delta 
			&\geq \log a_{\Pi}(x_T) \\
			&= \frac{\lambda}{\weight_{max}} \cdot \sum_{t=s}^{T-1}{\weight(x_T, x_{t+1})} - \frac{e^\lambda}{\weight_{max}} \sum_{t=s}^{T-1}{ \frac{1}{\drem _t(x_t)} } \\
			&\geq \left( \frac{1}{2} - \frac{1}{100} \right) \frac{\lambda}{\weight_{max}} - \frac{e^\lambda}{\weight_{max}} \sum_{t=s}^{T-1}{ \frac{1}{\drem _t(x_t)} }.
	\end{align*} 
	Rearranging the terms and using the fact that \(\weight_{max} \le \lambda / 50 \log \Delta\) gives the desired lower bound \begin{align*}
		\sum_{t=s}^{T-1}{ \frac{1}{\drem _t(x_t)} } 
			&\geq \left( \frac{1}{2} - \frac{1}{100} \right) \frac{\lambda}{e^\lambda} - 2 \frac{\weight_{max} \log \Delta}{ e^\lambda}  \\
			&\geq \left( \frac{1}{2} - \frac{1}{100} \right) \frac{\lambda}{e^\lambda} - \frac{\lambda}{25 e^\lambda} \\
			&\geq \frac{2\lambda}{5 e^\lambda}.
	\end{align*}
\end{proof}

We are now ready to prove Theorem~\ref{thm:weighted_general_weight_cycle_high}.

\begin{proof}[Proof of Theorem~\ref{thm:weighted_general_weight_cycle_high}]
	Let $\Pi = (x_0, \dots, x_T)$ be a path with $\wrem_{\Pi}(x_T) < 1/2$ and \[ 
	A_\Pi + B_\Pi \leq \frac{7 \lambda}{50 e^\lambda}\numberthis \label{eq:weighted_general_weight_cycle_high_proof_2},\] and let $s = act_\Pi(x_T)$. Note that the existence of $\Pi $ is guaranteed by Proposition~\ref{prop:weighted_general_ab_bounded}. First, observe that \begin{align}
		B_\Pi 
			&= \sum_{ t=0 }^{T-1}{ \left( \frac{1}{\drem _t(x_t)} e^{ -\frac{50 e^\lambda}{\lambda} \sum_{r=t}^{T-1}{ \weight(x_r, x_{r+1}) } } \right) } \nonumber \\
			&\ge \sum_{ t=s }^{T-1}{ \left( \frac{1}{\drem _t(x_t)} e^{ -\frac{50 e^\lambda}{\lambda} \sum_{r=t}^{T-1}{ \weight(x_r, x_{r+1}) } } \right) } \nonumber \\
			&\geq \left( \sum_{t=s}^{T-1}{ \frac{1}{\drem_t(x_t)} } \right) \cdot \left( e^{ -\frac{50 e^\lambda}{\lambda} \sum_{r=s}^{T-1}{ \weight(x_r, x_{r+1}) } } \right). \label{eq:weighted_general_weight_cycle_high_proof_1}
	\end{align}
 Here, the first inequality follows from considering only the terms of the sum from \(t=s\) to \(T-1\), and the second is obtained by adding terms from \(r=s\) to \(t-1\) in the sum in the exponential.

By Claim~\ref{lem:weighted_general_SIRD_cycle_high} we know that \begin{align}
		\sum_{t=s}^{T-1}{ \frac{1}{\drem_t(x_t)} } \geq \frac{2\lambda}{5 e^\lambda}. \label{eq:weighted_general_weight_cycle_high_proof_3}
	\end{align}
	Using \eqref{eq:weighted_general_weight_cycle_high_proof_2}, \eqref{eq:weighted_general_weight_cycle_high_proof_1} and \eqref{eq:weighted_general_weight_cycle_high_proof_3} we obtain \begin{align*}
		\frac{7 \lambda}{50 e^\lambda}
			&\geq \frac{2\lambda}{5 e^\lambda} \cdot e^{ -\frac{50 e^\lambda}{\lambda} \sum_{r=s}^{T-1}{ \weight(x_r, x_{r+1}) } },
	\end{align*} which by monotonicity of $\log$ yields \begin{align*}
		\sum_{t=s}^{T-1}{ \weight(x_t, x_{t+1}) }
			\geq \frac{\lambda}{50 e^\lambda} \log \frac{100}{35} \geq \frac{\lambda}{50 e^\lambda}.
	\end{align*}
	Since $s = act_P(x_T)$, we can simply take the edge $(x_T, x_s)$ to create the cycle $C = (x_s, \dots, x_T)$ which has weight \[\weight(C) = \weight(x_T, x_s) + \sum_{t=s}^{T-1}{\weight(x_t, x_{t+1})} \geq \frac{\lambda}{50 e^\lambda}.\]
\end{proof}


\section{Conclusion}
\label{sec:conclusion}
In this paper, we make improvements on two lines of questions in the literature concerning cycles in digraphs: decomposing Eulerian digraphs into few cycles, and finding heavy cycles in edge-weighted digraphs. We give the first non-trivial upper bound on the number of cycles we can decompose any Eulerian digraph into. 

For cycle decompositions, we propose the sum of inverse out-degrees of vertices in a cycle as a natural measure of its `weight' in a cycle decomposition. A central observation is Lemma \ref{lem:small_decomp_existence}, which links the existence of heavy cycles in Eulerian digraphs with the existence of a decomposition with few cycles. In Section \ref{sec:unweighted_logn_case}, we combine the natural greedy approach to find long cycles with a random walk to show that, under suitable minimum degree conditions, a digraph contains a cycle of weight $\Omega(1)$. In Section \ref{sec:weighted_case} we derandomize this technique in order to extend our result to more general weight functions. This modification allows us to improve the factor of $\log n$ to $\log \Delta$ in our main results. Note that our proofs can easily be turned into efficient algorithms to find heavy cycles and cycle decompositions.

By combining these ideas, we have shown that the number of cycles needed to decompose any $n$-vertex and $m$-edge Eulerian digraph is at most a factor $O(\log \Delta)$ away from the conjectured bound of $O(n)$. Looking closer at our analysis, the log-factor appears when iteratively removing cycles from $G$, as in Lemma \ref{lem:small_decomp_existence}, where $O(n \log \frac{m}{n})$ cycles are removed while the remaining graph has a high minimum degree, and an additional $O(n \log \Delta)$ cycles are removed while the minimum degree is low. It seems likely that the latter contribution can be decreased by a more careful consideration of graphs with low minimum degree. Note however that this would at most improve the log-factor from $\log\Delta$ to $\log \frac{m}{n}$.

Another potential way to improve upon our result could be to consider the structure of the intermediate graph after removing some number of cycles. On the one hand, the analysis in Section \ref{sec:unweighted_logn_case} shows that most of the edges in a heavy cycle can be chosen at random. On the other hand, it appears that the analysis in Section \ref{sec:unweighted_logn_case}, and in particular Lemma \ref{lem:visited_neighbours_bound} is wasteful for all but very particular graphs. It might also be possible to remove the $\log$ factor completely by modifying the approach to find all cycles simultaneously instead of iteratively. 

Regarding heavy cycles, there seem to be some possibilities for further line of research.
Another conjecture by Bollob\'as and Scott about heavy cycles in weighted Eulerian graphs somewhat combines Theorem~\ref{thm:main_result} and~\ref{thm:weighted_digraph_heavy_cycle_existence}. Our methods do not give any immediate results but we believe that refining our ideas might help solve the following.
\begin{conjecture}[Bollob\'as, Scott \cite{BS96}]\label{conj:eulerneednodegree}
Let $G$ be a weighted digraph such that every vertex $v\in V(G)$ satisfies $\win(v)=\wout(v)=1$. Then $G$ contains a cycle $C$ with $w(C)\ge 1$.
\end{conjecture}

 In the proof of Theorem~\ref{thm:main_result}, we use Corollary~\ref{cor:weighted_general_constant_weight_cycle_existence} which states that any digraph with sufficient minimal degree contains a cycle of constant weight. As the construction of Bollob\'as and Scott in \cite{BS96} shows, the minimal degree condition of $\delta \geq \Omega(\log \Delta)$ cannot be dropped. However, we believe that being \emph{Eulerian} is a sufficiently strong condition and that the same result, without the minimal degree condition, should be true on Eulerian digraphs.  

\begin{question}
	\label{conj:us}
	Is it true that if an edge-weighted Eulerian digraph $G$ is such that every vertex has $\wout(v) = 1$ then $G$ contains a cycle of weight at least $\Omega(1)$?
\end{question}
 As a first step it would be interesting to find graphs where our methods fail.
\begin{question}
Is there a Eulerian digraph $G$ such that the random walk in Section \ref{sec:unweighted_logn_case} finds a cycle with weight $o(1)$ with high probability? 
\end{question}
Indeed, Question~\ref{conj:us} can be seen as a weaker version of another conjecture proposed by Bollob\'as and Scott.

\begin{conjecture}[Bollob\'as, Scott \cite{BS96}]
	\label{conj:bollobas_scott_2}
	Let $G$ be a digraph with edge-weighting $\weight$ such that $\din(v) = \dout(v)$ for every vertex $v$ in $G$. Then $G$ contains a cycle of weight at least $c \weight(G) / (n-1)$, where $c$ is an absolute constant.
\end{conjecture}

We remark that it follows as a corollary of Theorem \ref{thm:main_result} that any digraph as above contains a cycle of weight $\Omega(\weight(G)/n \log \Delta)$, which confirms this conjecture up to a log-factor. As mentioned in the introduction, there exist edge-weightings of strongly connected digraphs where the weight of the longest cycle is only $O(\weight(G)/n^2)$. This confirms that requiring a digraph to have equal in- and out-degrees makes a significant difference for the existence of heavy cycles in edge-weighted graphs. 

\section*{Acknowledgements}

This project started at a workshop of the research group of Angelika Steger in Buchboden in July 2019. We would like to thank Milo\v{s} Truji\`c for suggesting the problem to us. We also thank the anonymous referees for their valuable feedback and for their suggestions. 

\bibliographystyle{abbrv}
\bibliography{long_directed_cycles}

\end{document}